\documentclass[noendauthors]{birkjour}

\usepackage{amssymb}
\usepackage{amsmath}  
\usepackage[thinc]{esdiff}
\usepackage{epstopdf}
\usepackage[caption=false]{subfig}
\usepackage[numbers,sort&compress]{natbib}
\bibpunct[, ]{[}{]}{,}{n}{,}{,}

\makeatletter
\def\NAT@def@citea{\def\@citea{\NAT@separator}}
\makeatother

\theoremstyle{plain}
 \newtheorem{thm}{Theorem}[section]
 \newtheorem{cor}[thm]{Corollary}
 \newtheorem{lem}[thm]{Lemma}
 
 \theoremstyle{definition}
 
 \theoremstyle{remark}

 \numberwithin{equation}{section}

\usepackage{enumerate}   
\usepackage{tikz}
\usetikzlibrary{shapes.geometric,arrows,decorations.markings,angles,calc}
\usepackage{relsize}
\usepackage{caption}

\usepackage[colorlinks=true, linkcolor=blue, citecolor=blue, urlcolor=blue]{hyperref}

\begin{document}

\def\nexto{\kern -0.54em}
\def\R{{\mathbb R}}
\def\T{{\mathbb T}}
\def\S{{\mathbb S}}
\def\C{{\mathbb C}}
\def\Z{{\mathbb Z}}
\def\N{{\mathbb N}}
\def\H{{\mathbb H}}
\def\B{{\mathbb B}}
\def\T{{\bf T}}
\def\diam{\mbox{\rm diam}}
\def\rr{{\cal R}}
\def\mt{{\Lambda}}
\def\e{\emptyset}
\def\sd{{\bf S}^{d-1}}
\def\stwo{{\bf S}^{2}}
\def\so{{\bf S}^1}
\def\dQ{\partial Q}
\def\dk{\partial K}
\def\endofproof{{\rule{6pt}{6pt}}}
\def\ts{\tilde{\sigma}}
\def\tr{\tilde{r}}
\def\be{\begin{equation}}
\def\ee{\end{equation}}
\def\beqn{\begin{eqnarray*}}
\def\eeqn{\end{eqnarray*}}
\def\ms{\medskip}
\def\beqn*{\begin{eqnarray*}}
\def\eeqn*{\end{eqnarray*}}
\def\endofproof{{\rule{6pt}{6pt}}}
\def\nv{\nabla \varphi}
\def\wuloc{W^u_{\mbox{\footnotesize\rm loc}}}
\def\wsloc{W^s_{\mbox{\footnotesize\rm loc}}}
\def\dist{\mbox{\rm dist}}
\def\diam{\mbox{\rm diam}}
\def\pr{\mbox{\rm pr}}
\def\supp{\mbox{\rm supp}}
\def\Arg{\mbox{\rm Arg}}
\def\In{\mbox{\rm Int}}
\def\Im{\mbox{\rm Im}}
\def\span{\mbox{\rm span}}
\def\spec{\mbox{\rm spec}\,}
\def\Re{\mbox{\rm Re}}
\def\var{\mbox{\rm var}}
\def\conf{\mbox{\footnotesize\rm const}}
\def\Conf{\mbox{\footnotesize\rm Const}}
\def\Lip{\mbox{\rm Lip}}
\def\con{\mbox{\rm const}\;}
\def\saa{\Sigma_A^+}
\def\sa{\Sigma_A}
\def\san{\Sigma^-_A}
\def\kmax{\kappa_{\max}}
\def\kmin{\kappa_{\min}}
\def\ecc{\mbox{\rm ecc}}
\def\tB{\widetilde{B}}
\def\hh{{\cal H}}
\def\thh{\widetilde{\cal H}}
\def\hE{\widehat{E}}
\def\ep{\epsilon}
\def\di{\displaystyle}
\def\pr{\mbox{\rm pr}}
\def\dK{\partial K}
\def\tpsi{\tilde{\psi}}
\def\tq{\tilde{q}}
\def\tsi{\tilde{\sigma}}
\def\ts{\tilde{s}}
\def\tW{\widetilde{W}}
\def\bs{\bigskip}
\def\tE{\widetilde{E}}
\def\tX{\widetilde{X}}
\def\yo{y^{(0)}}
\def\tY{\widetilde{Y}}
\def\tx{\tilde{x}}
\def\K{{\mathcal K}}
\def\T{{\mathcal T}}
\def\pp{{\mathcal P}}
\def\tgamma{\tilde{\gamma}}
\def\ot{(\omega,\theta)}
\def\overOm{\bar{\Omega}}
\def\spec{{\rm spec}}
\def\mm{{\mathcal M}}
\def\dk{\partial K}
\def\gen{{\rm generalized}\:}
\def\bic{{\rm bicharacteristic}}
\def\bics{{\rm bicharacteristics}}
\def\singsup{{\rm sing}\:{\rm supp}}
\def\Rnmin0{\R^n \setminus \{0\}}
\def\WF{WF}
\def\ff{{\mathcal F}}
\def\tu{\tilde{u}}
\def\tv{\tilde{v}}
\def\txi{\tilde{\xi}}
\def\teta{\tilde{\eta}}
\def\tSig{\widetilde{\Sigma}}
\def\tpsi{\tilde{\psi}}
\def\la{\langle}
\def\ra{\rangle}
\def\M{{\mathcal M}}
\def\aa{{\mathcal A}}
\def\tx{\tilde{x}}
\def\ty{\tilde{y}}
\def\Pm{P^{(m)}}
\def\tpi{\tilde{\pi}}
\def\ttheta{\tilde{\theta}}
\def\tOmega{\widetilde{\Omega}}
\def\Deone{\Delta^{(1)}}
\def\Detwo{\Delta^{(2)}}
\def\Dethree{\Delta^{(3)}}
\def\Defour{\Delta^{(4)}}
\def\tDeone{\widetilde{\Delta}^{(1)}}
\def\tDetwo{\widetilde{\Delta}^{(2)}}
\def\tDethree{\widetilde{\Delta}^{(3)}}
\def\tDefour{\widetilde{\Delta}^{(4)}}
\def\cc{{\mathcal C}}
\def\ta{\tilde{a}}
\def\tb{\tilde{b}}
\def\tM{\widetilde{M}}
\def\tf{\tilde{f}}
\def\Gr{\mbox{\rm Gr}}
\def\lan{\lambda^{(n)}}
\def\tn{t^{(n)}}
\def\En{E^{(n)}}
\def\Fn{F^{(n)}}
\def\lambdan{\lambda^{(n)}}
\def\tF{\widetilde{F}}
\def\tW{\widetilde{W}}
\def\ll{{\mathcal L}}
\def\tdelta{\tilde{\delta}}
\def\ti{\tilde{i}}
\def\hkappa{\hat{\kappa}}
\def\tkappa{\tilde{\kappa}}
\def\hL{\widehat{L}}
\def\tK{\widetilde{K}}
\def\tr{\tilde{r}}

\title [Some open billiards with distinct top Lyapunov exponents]{Some open billiards with distinct top \\ Lyapunov exponents}

\author[Amal Al Dowais]{Amal Al Dowais\textsuperscript{1}}
\author[Luchezar Stoyanov]{Luchezar Stoyanov\textsuperscript{2}}

\thanks{%
\textsuperscript{1}Department of Mathematics, College of Science and Arts,
Najran University, Najran, Saudi Arabia.
E-mail: \texttt{amaldowais@nu.edu.sa}}

\thanks{%
\noindent{\textsuperscript{2}Department of Mathematics and Statistics,
School of Physics, Mathematics and Computing,
University of Western Australia, Perth, WA 6009, Australia.
E-mail: \texttt{luchezar.stoyanov@uwa.edu.au}}}



\keywords{Lyapunov exponents, open billiard flow, non-wandering set.}
\subjclass{Primary: 37D20, 37D25; Secondary: 37D40}

\begin{abstract}
We study open billiard flows in $\R^3$ generated by reflections in the exterior of at least three ellipsoids
with $C^3$ boundaries. Each obstacle is congruent to
$$\tK = \{(x,y,z)\in \R^3 : x^2+y^2+(\beta z)^2 \leq r^2\},$$
with parameter $\beta>1$, small radii, and centers $O_i$ lying in the plane $\Gamma=\{z=0\}$.
The configuration satisfies the no-eclipse condition (H). For any Gibbs measure on the
non-wandering set of the billiard ball map, we prove that if $\beta>1/c_1^2$, for some constant
$c_1>0$ determined by the geometry of the configuration, the two positive Lyapunov exponents
are distinct and positive: $\lambda_1(\beta)>\lambda_2(\beta)>0$. Moreover, $\lambda_2$ is
independent of $\beta$ and corresponds to the restriction of the billiard flow to the plane $\Gamma$.
\end{abstract}

\maketitle

\section{Introduction}
\renewcommand{\theequation}{\arabic{section}.\arabic{equation}}

The study of Lyapunov exponents for billiard systems has a long history going back to the
foundational work of Sinai \cite{Si1,Si2}, who established hyperbolicity for dispersing
billiards in $\R^2$. For multi-dimensional dispersing billiards the structure of Lyapunov
spectra was studied by Chernov \cite{Ch1,Ch2,Ch3} and by B\'alint, Chernov, Sz\'asz and
T\'oth \cite{BCST}, among others. The relation between the geometry of the obstacles and the
Lyapunov spectrum has also been investigated in the physics literature, particularly in the
context of the Lorentz gas; see Bouchaud and Le Doussal \cite{BL}, Dellago and Posch \cite{DP},
and Latz, Van Beijeren and Dorfman \cite{LBS}. For open billiards -- where trajectories can
escape to infinity -- a natural framework is provided by the non-wandering set and its symbolic
dynamics, as developed in Petkov and Stoyanov \cite{PS}. The question of whether the two
positive Lyapunov exponents in $\R^3$ are distinct is of particular interest: it is related
to the existence of a dominated splitting of the unstable bundle and has implications for the
fine statistical properties of the billiard flow; see \cite{BCST} for related geometry.
The present paper is concerned with open billiards in $\R^3$ outside ellipsoidal obstacles,
and its immediate predecessor is our earlier work \cite{AmSt}, which treated the case of
spherical obstacles.

\ms

Let $K \subset \R^3$ be given by $K = K_1 \cup \cdots \cup K_{k_0},$
where $k_0 \geq 3$ and each $K_i$ is a strictly convex body in $\R^3$ with $C^3$ boundary
$\partial K_i$, with the $K_i$ pairwise disjoint. Set $\Omega = \R^3 \setminus K$. We impose
the following {\it no-eclipse condition} on $K$:
$${\rm (H)} \qquad
\text{For any distinct } i, j, \quad
{\rm conv}(K_i \cup K_j) \cap K_\ell = \varnothing
\quad \text{for all } \ell \neq i,j,$$
where ${\rm conv}(\cdot)$ denotes the convex hull. Under (H), the billiard flow $\phi_t$ on
the sphere bundle $S(\Omega) = \{ (q,v) : q \in \Omega, v \in \S^2 \}$ is called an
{\it open billiard flow}. Although $\phi_t$ exhibits singularities, its restriction to the
non-wandering set $M_0$ has only simple discontinuities at reflection points.

\ms

{\bf Assumptions.} For each $i$, we assume $K_i$ is an ellipsoid congruent to
$$\tK_i = \{ (x,y,z) \in \R^3 : x^2 + y^2 + (\beta z)^2 \leq r^2_i \} ,$$
where $r_i > 0$ is a small constant and $\beta > 1$ is a fixed parameter controlling the
eccentricity in the $z$-direction. Let $\Gamma = \{z=0\}$ in $\R^3$. We assume each $K_i$
is a translate of $\tK_i$, namely
$$K_i = \tK_i + O_i := \{ \tilde{x} + O_i : \tilde{x} \in \tK_i \},$$
where $O_i \in \Gamma$ is the centre of $K_i$. We further assume there exist constants
$d_0 > 0$ and $0 < \tr_0 < r_0 < d_0/3$ such that
$$d_0 < |O_i O_j| \leq 1 \quad \text{for all } i \neq j, \qquad \tr_0 \leq r_i \leq r_0.$$
The condition (H) implies that there exists a constant $c_1 > 0$ such that for any three
consecutive reflection points $q_1, q_2, q_3$ of a billiard trajectory in $\Omega$ and
outward unit normal $\nu(q_2)$ at $q_2$, the angle $\varphi$ between $\nu(q_2)$ and
$(q_3-q_2)/\|q_3-q_2\|$ satisfies $\cos\varphi \geq c_1$. Since $r_i \leq r_0 < d_0/3$,
consecutive reflection points satisfy
\begin{equation}
d_i = \|q_i - q_{i+1}\| \geq d_0 - 2r_0 > d_0/3. \label{eq:di}
\end{equation}

\bs

\noindent
{\bf Remark.} Clearly if $O_1, \ldots, O_{k_0} \in \Gamma$ with no three collinear, then
sufficiently small ellipsoids $K_i$ centred at $O_i$ automatically satisfy (H) and all
the assumptions above.

\bs

Since every centre $O_i$ lies in $\Gamma$ and each
$\tK_i$ is symmetric about $\Gamma$, the billiard is symmetric under the reflection
$(x,y,z)\mapsto(x,y,-z)$. Any trajectory with initial point $(q_0,v_0)$ satisfying
$q_0 \in \Gamma$ and $v_0 \in T\Gamma$ remains in $\Gamma$ for all time.
Conversely, a trajectory not starting in $\Gamma$ eventually escapes to infinity, by the
strict convexity argument of \cite{PS}. Hence the non-wandering set $M_0$ is contained
in $\Gamma$.

\ms

Let $\mu$ be an arbitrary Gibbs measure (therefore an ergodic measure) for the shift map
$\sigma$ on the related symbol space $\Sigma_A$ with $k_0$ symbols $\{1, \ldots, k_0\}$
(see Sect.~2). Using the natural isomorphism between $M_0$ and $\Sigma_A$ which conjugates
$B$ with $\sigma$ (see Sect.~2), we regard $\mu$ as a measure on $M_0$. By the
Multiplicative Ergodic Theorem (Theorem~2.1 below), the derivative cocycle $DB$ over $\mu$
has either one or two distinct positive Lyapunov exponents; we denote them
$\lambda_1(\beta) \geq \lambda_2(\beta) > 0$, where positivity is established in
Lemma~3.5 below. The unstable subspace $E^u(x)$ at $\mu$-a.e.\ $x \in M_0$ splits as
$$E^u(x) = E_0(x) \oplus F_0(x),$$
where $E_0(x) = \mathrm{span}\{n\}$ with $n=(0,0,1)$ is the subspace transverse to $\Gamma$,
and $F_0(x) \subset \Gamma$ is the in-plane component. The parameter $\beta$ enters the
curvature operator only through $E_0$ (via $\hkappa_i = \beta/r_i$), while $F_0$ involves
only $\kappa_i = 1/r_i$. Consequently, $\lambda_1(\beta) = \lambda(v)$ for $v \in E_0$
depends on $\beta$, whereas $\lambda_2 = \lambda(\tv)$ for $\tv \in F_0$ does not.

In a previous work \cite{AmSt} we considered a similar situation, where $K_i$ were balls
with small radii and centres in the plane $\Gamma$ and, under conditions similar to those
above, we proved that $\lambda_1 > \lambda_2$. In the case considered in \cite{AmSt} the
larger exponent $\lambda_1$ corresponds to the restriction of the billiard flow to the plane
$\Gamma$. In the case considered here the non-wandering set is still contained in $\Gamma$,
however now the smaller exponent $\lambda_2$ is the one determined by the restriction of the
billiard flow to $\Gamma$, while $\lambda_1$ reflects the additional stretching in the
transverse direction introduced by $\beta$.

\ms
In this note we prove the following:

\noindent
\begin{thm}
 Under the above assumptions, the billiard ball map $B$ on $M_0$
has exactly two positive Lyapunov exponents $\lambda_1(\beta) \geq \lambda_2(\beta) > 0$
with respect to any Gibbs measure $\mu$. If in addition $\beta > 1/c_1^2$, then the
following hold:
\begin{itemize}
\item[{\rm (i)}] The exponents are distinct: $\lambda_1(\beta) > \lambda_2(\beta) > 0$.
\item[{\rm (ii)}] The smaller exponent $\lambda_2(\beta) = \lambda_2$ is independent of
$\beta$ and equals the Lyapunov exponent of the restriction of the billiard flow to $\Gamma$.
\item[{\rm (iii)}] The gap satisfies $\lambda_1(\beta) \geq \lambda_2 + \log(\beta c_1^2)$.
\end{itemize}
\label{thm1.1}
\end{thm}


\section{Preliminaries}
\renewcommand{\theequation}{\arabic{section}.\arabic{equation}}

The setup closely follows \cite{AmSt}; we recall only what is needed for the proof.

\subsection{Billiard ball map and symbolic coding}

Set $M_1 = \{ (q,v) \in \partial K \times \S^2 : \langle v, \nu_K(q) \rangle > 0 \},$
where $\nu_K(q)$ is the outward unit normal to $\partial K$ at $q$. For $x \in M_1$, let
$B(x) = \phi_t(x)$ where $t > 0$ is the first return time to $M_1$. The {\it non-wandering
set} $M_0 \subset M_1$ consists of all $x$ whose full trajectory is bounded; on $M_0$ the
map $B$ is a homeomorphism, locally the restriction of a smooth map.

The map $B : M_0 \to M_0$ is conjugate to a transitive subshift of finite type. Let
$A = (A(i,j))$ be the $k_0 \times k_0$ matrix with $A(i,j)=1$ for $i \neq j$ and
$A(i,i)=0$. Define
$$\Sigma_A = \{ (i_j)_{j\in \Z} : 1 \leq i_j \leq k_0, \ A(i_j,i_{j+1})=1
\ \text{for all } j \},$$
with shift map $\sigma((i_j)) = (i_{j+1})$ and metric $d_\theta(z,w) = \theta^m$ where
$m = \max\{n : z_i = w_i \ \text{for } |i|<n\}$. The coding map $R : M_0 \to \Sigma_A$,
defined by $R(x) = (i_j)$ where $B^j(x) \in K_{i_j}$, is a homeomorphism satisfying
$R \circ B = \sigma \circ R$ (see \cite{PS}).

A Gibbs measure $\mu = \mu_\phi$ for a H\"older continuous $\phi$ on $(\Sigma_A, d_\theta)$
is the unique $\sigma$-invariant probability measure satisfying
\begin{equation}
C^{-1} \leq \frac{\mu(\cc[i_0,\dots,i_{n-1}])}{e^{S_n\phi(z)-nP}} \leq C
\end{equation}
for all cylinders $\cc[i_0,\dots,i_{n-1}]$ and $z \in \cc[i_0,\dots,i_{n-1}]$, where $P$
is the topological pressure and $S_n\phi = \sum_{k=0}^{n-1}\phi\circ\sigma^k$
(see \cite{B,PP}). Such measures are ergodic.

\subsection{Multiplicative Ergodic Theorem}

Let $f$ be an invertible ergodic map on $(\mathcal{X},\mu)$ and $L^n(x) =
L(f^{n-1}(x))\circ\cdots\circ L(x)$ a linear cocycle on $\R^d$. For background see
\cite{Ar,BP,V,BPS}.

\bs
\noindent
{\bf Theorem 2.1.} (Multiplicative Ergodic Theorem).
{\it There exists a full-measure subset $\ll \subset \mathcal{X}$ such that for all
$x \in \ll$:}

\ms

(a) {\it The limit $N(x) = \lim_{n\to\infty}(L^n(x)^t L^n(x))^{1/2n}$ exists.}

\ms

(b) {\it $N(x)$ has $k \geq 1$ distinct eigenvalues $t_k(x) < \cdots < t_1(x)$, and
there exist constants $t_k < \cdots < t_1$ such that $(t^{(n)}_i(x))^{1/n} \to t_i$,
where $t^{(n)}_i(x)$ is the $i$-th eigenvalue of $(L^n(x)^t L^n(x))^{1/2n}$.}

\ms

(c) {\it For each $j$, the eigenspace $E^{(n)}_j(x)$ of $(L^n(x)^t L^n(x))^{1/2n}$
converges to a limiting subspace $E_j(x)$ in $\mathrm{Gr}_{m_j}(\R^d)$, with
$\dim E_j(x) = m_j$ constant.}

\ms

(d) {\it For all $j = 1,\ldots,k$ and $u \in E_j(x)\setminus\{0\}$,
$$\lim_{n\to\pm\infty}\frac{1}{n}\log\|L^n(x)\cdot u\| = \log t_j.$$}

\bs

The numbers $\log t_k < \cdots < \log t_1$ are the (distinct) {\it Lyapunov exponents}
of $L$, and the $E_j(x)$ are the {\it Oseledets subspaces}, invariant under $L$.

Since $B : M_0 \to M_0$ is conjugate to a transitive subshift of finite type (Sect.~2.1),
Theorem~2.1 applies. For open billiards in $\R^3$ the unstable bundle $E^u(x)$ is
two-dimensional, giving either $k=2$ (equal positive exponents $\lambda_1 = \lambda_2$) or
$k=4$ (distinct positive exponents $0 < \lambda_2 < \lambda_1$, with $E_1(x), E_2(x)$
one-dimensional in $E^u(x)$ and negative exponents $-\lambda_1 < -\lambda_2 < 0$).

\section{Proof of Theorem 1.1}
\setcounter{equation}{0}

Our argument is based on the representation of the Poincar\'e map due to Petkov and
Vogel \cite{PV} (see \cite{PS}, Sect.~2.3), also used in \cite{AmSt} for spherical obstacles.
The key novelty is the ellipsoidal geometry, which introduces two distinct curvatures
$\kappa_i = 1/r_i$ (within $\Gamma$) and $\hat\kappa_i = \beta/r_i$ (orthogonal to $\Gamma$),
forcing the two Lyapunov exponents apart. We include all essential steps, referring to
\cite{AmSt} for parts that carry over unchanged.

\bs

Let $\mu=\mu_\phi$ be a Gibbs measure on $\sa$ defined by a H\"older continuous function
$\phi$. Denote by $\lambda_1 \geq \lambda_2 > 0$ the Lyapunov exponents of the billiard ball
map $B$ on the non-wandering set $M_0$ with respect to $\mu$ (here we identify $M_0$ with
$\sa$ via the coding map $R$; see Sect.~2.1). Let $\Sigma_0 \subset \sa$ correspond to $\ll$
under the natural conjugacy between $\sa$ and $\mathcal{X}$.

Fix $\tx_0=(p_0,v_0) \in M_0$ with $R(\tx_0)\in \Sigma_0$, let $\gamma$ be its billiard
trajectory with successive reflection points $q_1,q_2,\dots$, and assume it is never
perpendicular to $\partial K$ at a reflection point. For $i \geq 1$, let $\Pi_i$ be the
plane through $q_i$ orthogonal to $\overrightarrow{q_iq_{i+1}}$, $\omega_i$ the unit
vector along $\overrightarrow{q_iq_{i+1}}$, $\alpha_i=T_{q_i}(\partial K)$,
$\sigma_i$ the reflection in $\alpha_i$, $\nu_i$ the outward unit normal at $q_i$,
$d_i=\|q_{i-1}-q_i\|$, and $\theta_i=2\langle\nu_i,\omega_i\rangle \geq 2c_1 > 0$.

The full derivation of the coordinate system, the maps $\Phi_{i+1}$, and the local conjugacy
between $\Phi = \Phi_m\circ\cdots\circ\Phi_1$ and the billiard flow follows exactly as in
\cite{AmSt}, Sect.~3 (see also \cite{PS}, Sect.~2.3). The outcome is that the Lyapunov
exponent $\lambda(v)$ at $x_0$ in direction $v \in E^u(q_0)$ can be expressed as
\begin{equation}
\lambda(v) = \lim_{m\to\infty}\,\frac{1}{m}\log\|\pr_2(\Pm_\gamma(0,v))\| , \label{eq:lyap}
\end{equation}
where $\pr_2(u',v')=v'$ and $\Pm_\gamma : \Pi_0\times\Pi_0 \to \Pi_m\times\Pi_m$ is the
linearised Poincar\'e map described in Lemma~3.1 below. Define $s_i$ and $\psi_i$ as in
\cite{AmSt}, equations (3.5)--(3.6).

\bs

\begin{lem}\label{lem:poincare}
{\it The map $\Pm_\gamma : \Pi_0 \times \Pi_0 \longrightarrow \Pi_m \times \Pi_m$
is a linear symplectic map with matrix representation
\be
\Pm_\gamma =
 \left(
\begin{array}{cc}
s_m & 0\\
0  & s_m
\end{array}
\right) \cdot
\left(
\begin{array}{cc}
I & d_m I\\
\psi_m  & I + d_m \psi_m
\end{array}
\right)
\cdots
\left(
\begin{array}{cc}
I & d_1 I\\
\psi_1  & I + d_1 \psi_1
\end{array}
\right).
\ee
}
\end{lem}

The proof is identical to that in \cite{AmSt} (see also \cite{PS}, Theorem.~2.3.1).

\bs

We now describe the key modification introduced by the ellipsoidal geometry. Since all
centres $O_i$ lie in $\Gamma$, every trajectory in $M_0$ is contained in $\Gamma$, all
$\omega_i, \nu_i$ lie in $\Gamma$, and each $\Pi_i$ is perpendicular to $\Gamma$.
Set $\kappa_i=1/r_i$ and $\hat\kappa_i=\beta/r_i$. The second fundamental form $G_i$ of
$\partial K_i$ at $q_i$ satisfies
\begin{equation}
G_i(\xi)=\kappa_i \,\xi \ \text{ for } \xi \in \alpha_i \cap \Gamma, \qquad
G_i(\eta)=\hat\kappa_i \,\eta \ \text{ for } \eta\in \alpha_i,\ \eta\perp \Gamma,
\label{eq:Gi}
\end{equation}
the only difference from \cite{AmSt} where $\kappa_i = \hat\kappa_i$.
Let $E_i = \Pi_i \cap \mathrm{span}(n)$ (where $n=(0,0,1)$), $F_i = \Pi_i \cap \Gamma$,
and $f_i$ the unit vector spanning $F_i$. Using \eqref{eq:Gi} and the symmetry map $\tpsi_i$
(see \cite{AmSt}, Sect.~3):
\begin{equation}
(\tpsi_i)_{|E_i} = \hat\kappa_i\theta_i\,I , \qquad
\tpsi_i(f_i) = \kappa_i\,\ttheta_i\,f_i ,
\label{eq:psi}
\end{equation}
where $\ttheta_i = L_i\theta_i$, $L_i = 1/\langle\omega_i,\nu_i\rangle^2$, and
\begin{equation}
1 \leq L_i \leq L_0 = 1/c_1^2. \label{eq:Li}
\end{equation}
{\bf We will assume $\beta > L_0 = 1/c_1^2$.}

\bs

From \eqref{eq:psi} the Poincar\'e map decouples:
$\Pm_\gamma : E_0\times E_0 \to E_m\times E_m$ and
$\Pm_\gamma : F_0\times F_0 \to F_m\times F_m$.
Setting $a_i=\hat\kappa_i\theta_i$, $\ta_i=\kappa_i\ttheta_i$,
$b_i=1+\hat\kappa_i d_i\theta_i$, $\tb_i=1+\kappa_i d_i\ttheta_i$,
we have $a_i > \ta_i$ and $b_i > \tb_i$ since $\beta > L_0 \geq L_i$.
With the block structure of Lemma~\ref{lem:poincare}, the restrictions of $\Pm_\gamma$ to
$E_0\times E_0$ and $F_0\times F_0$ take the form \eqref{eq:lyap} with $(a_i,b_i)$ and $(\ta_i,\tb_i)$
respectively.

\ms

The following lemma was proved in \cite{AmSt} (see Lemma~3.2 there).

\bs

\begin{lem}\label{lem:scalar}
{\it For every $j = 1,2, \ldots, m$ there exist scalar functions
$\Deone_j, \Detwo_j, \Dethree_j,$ $ \Defour_j$ (resp.\ $\tDeone_j, \tDetwo_j, \tDethree_j,
\tDefour_j$), which are polynomials in $\hat\kappa_i, d_i, \theta_i$
(resp.\ $\kappa_i, \ttheta_i, d_i$) with positive coefficients, such that
\begin{equation}
\prod_{i=j}^{1}
\begin{pmatrix} I & d_i I \\ a_i I & b_i I \end{pmatrix}
=
\begin{pmatrix} \Deone_j I & \Detwo_j I \\ \Dethree_j I & \Defour_j I \end{pmatrix},
\qquad
\prod_{i=j}^{1}
\begin{pmatrix} I & d_i I \\ \ta_i I & \tb_i I \end{pmatrix}
=
\begin{pmatrix} \tDeone_j I & \tDetwo_j I \\ \tDethree_j I & \tDefour_j I \end{pmatrix}.
\end{equation}
}
\end{lem}

Combining \eqref{eq:lyap} with the representation of Lemma~\ref{lem:scalar}, we obtain the following.

\bs

\begin{cor}\label{cor:lyap}
{\it For all $v \in E_0$ and $\tv \in F_0$ with $\|v\| = \|\tv\| = 1$,
\be
\lambda(v) = \lim_{m\to\infty} \frac{1}{m} \log \Defour_m \quad , \quad
\lambda(\tv) = \lim_{m\to\infty} \frac{1}{m} \log \tDefour_m .
\ee
}
\end{cor}

\ms

The recurrences for these quantities are (see \cite{AmSt}):
\begin{equation}
\Detwo_{j+1} = \Detwo_j + d_{j+1}\, \Defour_j \quad, \quad
\Defour_{j+1} = a_{j+1}\,\Detwo_j + b_{j+1}\,\Defour_j , \label{eq:recA}
\end{equation}
\begin{equation}
\tDetwo_{j+1} = \tDetwo_j + d_{j+1}\,\tDefour_j \quad, \quad
\tDefour_{j+1} = \ta_{j+1}\,\tDetwo_j + \tb_{j+1}\,\tDefour_j , \label{eq:recB}
\end{equation}
with $\Detwo_0 = \tDetwo_0 = \Defour_0 = \tDefour_0 = 1$.

\bs

\begin{lem}\label{lem:delta}
{\it For all $j \geq 1$,
$\Detwo_j \leq \Defour_j$ and $\tDetwo_j \leq \tDefour_j$.}
\end{lem}

\ms

\noindent
{\it Proof.} Same as in \cite{AmSt}. \endofproof

\bs

We have $1 < 1/r_0 \leq \kappa_i \leq 1/\tr_0$,
$2c_1 \leq \theta_i \leq 2$ and $d_0/3 \leq d_i \leq d_0 < 1$ for all $i$. Setting
$$s = s(\beta) = \min\left\{ \log \left( \frac{\beta d_0 \tr_0}{6c_1} \right)\;,\,
\log (\beta c_1^2) \right\} > 0,$$
one verifies (as in \cite{AmSt}) that
\begin{equation}
b_i \geq e^s\, \tb_i \quad \text{and} \quad a_i \geq e^s\, \ta_i \quad \text{for all } i \geq 1. \label{eq:key}
\end{equation}

\bs

\begin{lem}\label{lem:inductive}
{\it There exists a sequence $\{t_i\}_{i\geq 1}$ of non-negative numbers
such that, setting $u_1 = 1$, $v_1 = e^s$, and
$$u_{i+1} = e^{t_{i+1}}\,u_i \quad,\quad v_i = e^s\,u_i,$$
we have $u_i \geq 1$ for all $i \geq 1$,
\begin{equation}
\Detwo_i \geq u_i\,\tDetwo_i \quad \text{and} \quad \Defour_i \geq v_i\,\tDefour_i \label{eq:uv}
\end{equation}
for all $i \geq 1$, and $0 \leq t_{i+1} < s$ for all $i \geq 1$.}
\end{lem}

\ms

\noindent
{\it Proof.}  We will do the construction by induction. Set $u_1 = 1$ and $v_1 = e^s$.

From \eqref{eq:recA} and \eqref{eq:recB} we have $\Detwo_1 = d_1 = \tDetwo_1$ and, using \eqref{eq:key},
$\Defour_1 = b_1 \geq e^s \tb_1 = v_1 \tDefour_1$.
So \eqref{eq:uv} holds for $i = 1$.

Assume that $u_i$, $t_i$, $v_i = e^s u_i$ have been constructed for some $i \geq 1$
so that \eqref{eq:uv} holds. We will now define $u_{i+1}$, $t_{i+1}$ and $v_{i+1} = e^s u_{i+1}$.

Using \eqref{eq:recA} and the inductive hypothesis \eqref{eq:uv} we get
$$\Detwo_{i+1} = \Detwo_i + d_{i+1}\,\Defour_i
\geq u_i\,\tDetwo_i + d_{i+1}\, e^s u_i\,\tDefour_i.$$
From \eqref{eq:recB}, $d_{i+1}\,\tDefour_i = \tDetwo_{i+1} - \tDetwo_i$, so the above yields
\begin{equation}
\Detwo_{i+1} \geq u_i\bigl(e^s\,\tDetwo_{i+1} - (e^s - 1)\tDetwo_i\bigr). \label{eq:D2est}
\end{equation}
From \eqref{eq:recB} and Lemma~\ref{lem:delta}, $\tDetwo_{i+1} \geq (1+d_0)\,\tDetwo_i$, hence
$\tDetwo_i \leq \tDetwo_{i+1}/(1+d_0)$. Substituting into \eqref{eq:D2est} gives
$$\Detwo_{i+1} \geq u_i\left(\frac{d_0 e^s + 1}{1+d_0}\right)\tDetwo_{i+1}.$$
Define $t_{i+1} > 0$ by
\begin{equation}
t_{i+1} = \log\frac{d_0 e^s + 1}{1+d_0} < s, \label{eq:ti}
\end{equation}
and set $u_{i+1} = e^{t_{i+1}} u_i \geq u_i \geq 1$.
Then $\Detwo_{i+1} \geq u_{i+1}\,\tDetwo_{i+1}$, so \eqref{eq:uv} holds for $\Detwo_{i+1}$.

Now set $v_{i+1} = e^s u_{i+1}$. We check that \eqref{eq:uv} holds for $\Defour_{i+1}$.
It follows from \eqref{eq:recA}, \eqref{eq:key} and the inductive assumptions that
\begin{eqnarray*}
\Defour_{i+1}
&=& a_{i+1}\,\Detwo_i + b_{i+1}\,\Defour_i
\geq e^s\ta_{i+1}\,\Detwo_i + e^s\tb_{i+1}\,\Defour_i\\
&\geq& e^s\bigl(\ta_{i+1}\,u_i\,\tDetwo_i + \tb_{i+1}\, e^s u_i\,\tDefour_i\bigr)\\
&=& e^s u_{i+1}\bigl(\ta_{i+1}\, e^{-t_{i+1}}\,\tDetwo_i
   + \tb_{i+1}\, e^{s-t_{i+1}}\,\tDefour_i\bigr).
\end{eqnarray*}
We will check now that $v_{i+1}\,\tDefour_{i+1} = e^s u_{i+1}\,\tDefour_{i+1}$ satisfies
$$e^s u_{i+1}\bigl(\ta_{i+1}\, e^{-t_{i+1}}\,\tDetwo_i
+ \tb_{i+1}\, e^{s-t_{i+1}}\,\tDefour_i\bigr) \geq v_{i+1}\,\tDefour_{i+1},$$
that is,
\begin{equation}
\frac{\tb_{i+1}\,\tDefour_i}{\ta_{i+1}\,\tDetwo_i}
\geq \frac{e^{t_{i+1}}-1}{e^s - e^{t_{i+1}}}. \label{eq:ratio}
\end{equation}
The definition \eqref{eq:ti} and a simple calculation give
$$\frac{e^{t_{i+1}} -1}{e^s - e^{t_{i+1}}} = \frac{ \frac{d_0 e^s + 1}{1+d_0} -1}{e^s -  \frac{d_0 e^s + 1}{1+d_0}}
= \frac{d_0 e^s + 1 -1 -d_0}{e^s + d_0 e^s - d_0 e^s -1} = \frac{d_0 (e^s-1)}{e^s-1} = d_0 .$$
On the other hand, for the left-hand side of \eqref{eq:ratio}, it follows from Lemma~\ref{lem:delta} that
$$\frac{\tb_{i+1}\,\tDefour_i}{\ta_{i+1}\,\tDetwo_i}
\geq \frac{\tb_{i+1}}{\ta_{i+1}}
= \frac{1 + \kappa_{i+1} d_{i+1} L_{i+1}\theta_{i+1}}{\kappa_{i+1} L_{i+1}\theta_{i+1}}
> d_{i+1} \geq d_0,$$
since $d_{i+1} \geq d_0$. This proves \eqref{eq:ratio}. The latter in turn implies
$\Defour_{i+1} \geq e^s u_{i+1}\,\tDefour_{i+1} = v_{i+1}\,\tDefour_{i+1}$,
and therefore \eqref{eq:uv} holds with $i$ replaced by $i+1$.

This completes the induction.
\endofproof

\bs

We will now show that both $\lambda(v)$ and $\lambda(\tv)$ are strictly positive.

\begin{lem}\label{lem:positive}
{\it We have $\lambda(v) > 0$ and $\lambda(\tv) > 0$.}
\end{lem}

\ms

\noindent
{\it Proof.} We first show $\lambda(\tv) > 0$. Since $\ta_{i+1}, \tDetwo_i \geq 0$, the
recurrence \eqref{eq:recB} gives
$$\tDefour_{i+1} = \ta_{i+1}\,\tDetwo_i + \tb_{i+1}\,\tDefour_i \geq \tb_{i+1}\,\tDefour_i.$$
Using $\kappa_i \geq 1/r_0$, $d_i \geq d_0/3$ (from \eqref{eq:di}), $L_i \geq 1$, and
$\theta_i \geq 2c_1$:
$$\tb_i = 1 + \kappa_i d_i L_i\theta_i \geq 1 + \frac{2c_1 d_0}{3r_0}.$$
Setting $c_0 = \log\!\left(1 + \frac{2c_1 d_0}{3r_0}\right) > 0$, it follows by induction that
$$\tDefour_m \geq e^{c_0 m}\,\tDefour_0 = e^{c_0 m} \quad \text{for all } m \geq 1,$$
since $\tDefour_0 = 1$. Therefore $\lambda(\tv) = \lim_{m\to\infty}\frac{1}{m}\log\tDefour_m
\geq c_0 > 0$.

The same argument applies to $\lambda(v)$, replacing $\tb_i$ by $b_i \geq \tb_i$ throughout
(since $\hat\kappa_i = \beta\kappa_i > \kappa_i L_i$ as $\beta > L_0 \geq L_i$). Hence
$\Defour_m \geq e^{c_0 m}$ for all $m \geq 1$, and $\lambda(v) \geq c_0 > 0$.
\endofproof

\bs

{\bf {To finish the proof of Theorem \ref{thm1.1}}}, let $v \in E_0$ and $\tv \in F_0$ with
$\|v\| = \|\tv\| = 1$. Then \eqref{eq:lyap} and \eqref{eq:uv} (with $v_m = e^s u_m \geq e^s$) imply
\begin{eqnarray*}
\lambda(v)
& = & \lim_{m\to\infty} \frac{1}{m} \log \Defour_m
  \geq \lim_{m\to\infty} \frac{1}{m} \log (e^{ms}\,\tDefour_m)\\
& = & \lim_{m\to\infty} \frac{1}{m} \left(ms + \log\tDefour_m\right)
  = s + \lambda(\tv).
\end{eqnarray*}

Since $s \geq \log(\beta c_1^2) > 0$, we have $\lambda(v) > \lambda(\tv)$, so we can now
identify $\lambda_1(\beta) = \lambda(v)$ and $\lambda_2 (\beta) = \lambda(\tv)$. Since $\lambda(\tv)$ depends only on the dynamics restricted to $\Gamma$, it is
independent of $\beta$, so $\lambda_2(\beta) = \lambda_2$. By Lemma~\ref{lem:positive},
$\lambda_2 = \lambda(\tv) > 0$, and the gap satisfies
$$\lambda_1(\beta) - \lambda_2 \geq s \geq \log(\beta c_1^2) > 0,$$
establishing $\lambda_1 (\beta) \geq \lambda_2 + \log(\beta c_1^2) > \lambda_2 > 0$.  This completes the proof of Theorem~1.1.
\endofproof

\bs

\section*{Disclosure statement}
The authors report there are no competing interests to declare.

\section*{Data Availability Statement}
Data sharing not applicable to this article as no datasets were generated or analysed
during the current study.

\end{document}